\documentclass[11pt,oneside,english]{amsart}
\usepackage[T1]{fontenc}
\usepackage[latin9]{inputenc}
\usepackage{geometry}
\usepackage{amstext}
\usepackage{amsthm}
\usepackage{amssymb}

\makeatletter
\theoremstyle{plain}
\newtheorem{thm}{\protect\theoremname}
\theoremstyle{plain}
\newtheorem{prop}[thm]{\protect\propositionname}
\theoremstyle{plain}
\newtheorem{lem}[thm]{\protect\lemmaname}

\usepackage{lmodern}

\makeatother

\usepackage{babel}
\providecommand{\lemmaname}{Lemma}
\providecommand{\propositionname}{Proposition}
\providecommand{\theoremname}{Theorem}

\usepackage{lineno,hyperref}
\usepackage{amssymb,amsmath,amsthm,latexsym,xcolor}
\usepackage{mathtools}

\usepackage{enumitem}

\modulolinenumbers[1]
\let\oldequation\equation
\let\oldendequation\endequation

\renewenvironment{equation}
  {\linenomathNonumbers\oldequation}
  {\oldendequation\endlinenomath}
\definecolor{orange}{rgb}{0.7,0.3,0}
\definecolor{blue}{rgb}{.2,.6,.75}
\definecolor{green}{rgb}{.4,.7,.4}
\definecolor{purple}{RGB}{127,0,255}
\newcommand{\abs}[1]{\lvert #1 \rvert}
\newcommand{\bigabs}[1]{\Bigl \lvert #1 \Bigr \rvert}
\begin{document}

%\begin{frontmatter}

\title[The distribution of spacings of real-valued lacunary sequences mod one]{The distribution of spacings of real-valued lacunary sequences modulo one}
%\tnotetext[mytitlenote]{Fully documented templates are available in the elsarticle package on \href{http://www.ctan.org/tex-archive/macros/latex/contrib/elsarticle}{CTAN}.}

%% Group authors per affiliation:
\author{Sneha Chaubey}
\address{Department of Mathematics, Indraprastha Institute of Information Technology, New Delhi 110020, India}
%\corref{mycorrespondingauthor}
%\cortext[mycorrespondingauthor]{Corresponding author}
\email{sneha@iiitd.ac.in}
\author{Nadav Yesha}
\address{Department of Mathematics, University of Haifa,
	Haifa 3498838,	Israel}
\email{nyesha@univ.haifa.ac.il}

\thanks{SC is supported by the Science and Engineering Research Board, Department of Science and Technology, Government of India under grant SB/S2/RJN-053/2018. NY is supported by the ISRAEL SCIENCE FOUNDATION (grant No. 1881/20).}
%\footnote{SC is supported by the Science and Engineering Research Board, Department of Science and Technology, Government of India under grant SB/S2/RJN-053/2018.}
%\footnote{NY is supported by the ISRAEL SCIENCE FOUNDATION (grant No. 1881/20).}
\begin{abstract}
Let $\left(a_{n}\right)_{n=1}^{\infty}$ be a lacunary sequence of
positive real numbers. Rudnick and Technau showed that for almost
all $\alpha\in\mathbb{R}$, the pair correlation of $\left(\alpha a_{n}\right)_{n=1}^{\infty}$
mod 1 is Poissonian. We show that all higher
correlations and hence the nearest-neighbour spacing distribution
are Poissonian as well, thereby extending a result of Rudnick and
Zaharescu to real-valued sequences.
\end{abstract}

%\begin{keyword}
%
%\MSC[2010] xx-yy\sep  xx-yy
%\end{keyword}

%\end{frontmatter}

%\linenumbers
\maketitle
\section{Introduction}
A sequence $\left(a_{n}\right)_{n=1}^{\infty}$ of real numbers is
said to be \emph{uniformly distributed modulo one} (u.d.\ mod 1), if
the fractional parts of the sequence are equidistributed in the unit
interval, i.e., for every interval $I\subseteq[0,1)$ we have
\[
\lim_{N\to\infty}\frac{1}{N}\#\left\{ 1\le n\le N:\,\left\{ a_{n}\right\} \in I\right\} = |I|.
\]

Questions about u.d.\ mod 1 have been studied for more than a century
now, going back to the pioneering work of Weyl \cite{Weyl}. Interestingly,
from a metric point of view, the conditions for u.d.\ mod $1$ are
quite modest: as was shown by Weyl \cite{Weyl}, if $\left(a_{n}\right)$ is any
sequence of \emph{distinct integers,} then the sequence $\left(\alpha a_{n}\right)$
is u.d.\ mod $1$ for almost all $\alpha\in\mathbb{R}$. It is also
well-known (see, e.g., \cite[Corollary 4.3]{KN}) that if $\left(a_{n}\right)$ is \emph{real-valued}
and is sufficiently well-spaced in the sense that there exists $\delta>0$
such that $\left|a_{n}-a_{m}\right|\ge\delta$ for all $n\ne m$,
then $\left(\alpha a_{n}\right)$ is u.d.\ mod $1$ for almost all
$\alpha\in\mathbb{R}$. The latter condition clearly holds when $\left(a_{n}\right)$
is a \emph{lacunary} sequence of positive real numbers, i.e., when
there exists $c>1$ such that $a_{n+1}/a_{n}\ge c$ for all $n\ge1$.

While very useful, the notion of u.d.\ mod $1$ cannot explain the
finer aspects of sequences modulo one, such as the \emph{pseudo-randomness}
of a sequence. Indeed, there is a growing interest in studying \emph{fine-scale}
statistics of sequences modulo one in the scale of the mean gap $1/N$;
one can test for pseudo-randomness by comparing these statistics to those
of random, uniformly distributed independent points in the unit interval
(Poisson statistics). A most natural statistic, which is very easy
to visualize, is the \emph{nearest-neighbour spacing distribution}
(or gap distribution), which is defined as follows: consider the \emph{ordered}
fractional parts $\left\{ a_{n}\right\} $ of the first $N$ elements
of the sequence, which we denote by
\[
a_{\left(1\right)}^{N}\le a_{\left(2\right)}^{N}\le\dots\le a_{\left(N\right)}^{N},
\]
and denote $a_{\left(N+1\right)}^{N}:=1+a_{\left(1\right)}^{N}$.
Let the normalized gaps be defined by
\[
\delta_{n}^{N}:=N\left(a_{\left(n+1\right)}^{N}-a_{\left(n\right)}^{N}\right);
\]
we say that nearest-neighbour spacing distribution is Poissonian if
for any $I\subseteq[0,\infty)$,
\[
\lim_{N\to\infty}\frac{1}{N}\#\left\{ 1\le n\le N:\,\delta_{n}^{N}\in I\right\} =\int_{I}e^{-s}\,ds,
\]
i.e., if the limit distribution agrees with the random model.

There are very few examples of sequences modulo one where a Poissonian
nearest-neighbour spacing distribution can be rigorously proved. Rudnick
and Zaharescu proved \cite{RZ2} that if $\left(a_{n}\right)$ is a
lacunary sequence of \emph{integers,} then the nearest-neighbour spacing
distribution of $\left(\alpha a_{n}\right)$ is Poissonian for almost
all $\alpha\in\mathbb{R};$ as will be detailed below, the main goal
of this note is to show an analogous result for \emph{real-valued}
lacunary sequences. Another natural question about real-valued lacunary
sequences with a different notion of randomization (which dramatically changes the problem) is whether the sequence $\left(\alpha^{n}\right)$ has a Poissonian
nearest-neighbour spacing distribution for almost all $\alpha>1$ --
this was recently answered positively in \cite{ABTY}, as a special case
of a more general family of sequences (which include some sub-lacunary
sequences as well) having this property. For polynomially growing sequences
very little is known. Rudnick and Sarnak conjectured \cite{RS} that for any
$d\ge2$ and any $\alpha$ which cannot be approximated too well by
rationals, the sequence $(\alpha n^{d})$ has a Poissonian
nearest-neighbour spacing distribution; while numerical experiments
are supportive of the conjecture (and of Poisson statistics for
many other natural examples of sequences), it remains open until today.

A related class of important fine-scale statistics consists of the
$k$-level correlations ($k\ge2$). Given a compactly supported, real
valued, smooth function $f:\mathbb{R}^{k-1}\to\mathbb{R}$, we define
the $k$-level correlation sum to be
\[
R_{k}\left(f,\left(a_{n}\right),N\right):=\frac{1}{N}\sum_{m\in\mathbb{Z}^{k-1}}\sum_{\substack{x=(x_1,\dots,x_k)\\
		1\le x_{1},\dots,x_{k}\le N\\
		x_{1},\dots,x_{k}\,\text{distinct}
	}
}f\left(N\left(\Delta_{(a_n)}(x)-m\right)\right),
\]
where $$\Delta_{(a_n)}(x):=\left(a_{x_{1}}-a_{x_{2}},a_{x_{2}}-a_{x_{3}},\dots,a_{x_{k-1}}-a_{x_{k}}\right).$$
We say that $\left(a_{n}\right)$ has Poissonian $k$-level correlation,
if for any compactly supported, real valued, smooth function $f:\mathbb{R}^{k-1}\to\mathbb{R}$
we have
\begin{equation}
	\lim_{N\to\infty}R_{k}\left(f,\left(a_{n}\right),N\right)=\int_{\mathbb{R}^{k-1}}f\left(x\right)\,dx,\label{eq:Poisson_k_correlation}
\end{equation}
which again agrees with the random model. It is well-known (see, e.g.,
\cite[Appendix A]{KR}), that Poissonian $k$-level correlations for all $k\ge2$
implies Poissonian nearest-neighbour spacing distribution (it also
implies Poissonian behaviour for other statistics, such as the second-to-nearest-neighbour
spacing distribution, joint nearest-neighbour spacing distribution,
etc.).% Thus, a standard approach for proving Poissonian nearest-neighbour
%spacings, also taken by us in this note, is to prove (\ref{eq:Poisson_k_correlation})
%for all $k\ge2.$

In many instances, although one fails to obtain information on the
triple and higher correlations of a sequence, one can still study
the pair correlation ($k=2$), and prove a Poissonian limit.
Rudnick and Sarnak showed \cite{RS} that for any $d\ge2$, the sequence $(\alpha n^{d})$
has Poissonian pair correlation for almost all $\alpha\in\mathbb{R}$.
Recently, a significant progress was made in the study of the sequences $(\alpha n^{\theta})$
where $\theta$ is non-integer: Aistleitner, El-Baz and Munsch proved \cite{AEM} Poissonian pair
correlation for any fixed $\theta>1$ and almost $\alpha\in\mathbb{R};$
this was recently extended by Rudnick and Technau \cite{RT2} to all
fixed $\theta<1$ and almost all $\alpha\in\mathbb{R}$. As for non-metric
results, El-Baz, Marklof and Vinogradov showed \cite{EMV} Poissonian pair correlation for the sequence $\left(\sqrt{n}\right)_{\sqrt{n}\notin\mathbb{Z}}$
(the nearest-neighbour spacing distribution of this sequence
is \emph{non-Poissonian,} see \cite{EM}); recently, Lutsko, Sourmelidis and Technau  proved \cite{LST}
Poissonian pair correlation for the sequence $(\alpha n^{\theta})$
for all fixed $\alpha\ne0$ and $\theta<14/41=0.341\dots$.

In \cite{RT1}, Rudnick and Technau proved that for any \emph{real-valued,}
positive lacunary sequence $\left(a_{n}\right)$, the pair correlation
of $\left(\alpha a_{n}\right)$ is Poissonian for almost all $\alpha>1$, extending a result
of Rudnick and Zaharescu \cite{RZ1} holding only for \emph{integer-valued}
sequences. Our goal is to show metric Poisson behaviour for the higher-level
correlations ($k\ge3)$ of this sequence.
\begin{thm}
	\label{thm:MainThm}Let $\left(a_{n}\right)$ be a real-valued, positive, lacunary sequence.
	For almost all $\alpha\in\mathbb{R}$, the $k$-level correlation
	of $\left(\alpha a_{n}\right)$ is Poissonian for all $k\ge2.$
\end{thm}

In particular, we conclude that for almost all $\alpha\in\mathbb{R}$,
the nearest-neighbour spacing distribution (and all other statistics determined by the correlations) of $\left(\alpha a_{n}\right)$
is Poissonian.

\subsection{Outline of the proof}

In Section \ref{sec:counting} we give a combinatorial counting argument which closely follows the argument for integer-valued sequences from \cite[Section 2]{RZ2}, with several adaptations required to extend the proof to real-valued sequences. In Section \ref{sec:proof} we use the bound from the previous section to prove a polynomial decay for the variance of the $k$-level correlation sum, which by a standard argument gives the claimed almost sure convergence.

\subsection{Notation}
Throughout this note, we will interchangeably use the Bachmann-Landau $ O $ notation and the Vingoradov notation $ \ll $, where for readability reasons the implied constants will be omitted, and may depend on $ a_1 $ (the first element of the sequence), the constant $ c $ defined below, the parameters $ r,k,\epsilon,\eta,R $ and the functions $ f,\rho $.

\section{A counting argument}
\label{sec:counting}
Let $(a_n)_{n=1}^\infty$ be a lacunary sequence of positive
\emph{real} numbers, i.e., $a_1>0$, and there exists a constant $c>1$ such that
\begin{equation}
a_{n+1}\ge ca_n\label{eq:lacunarity}
\end{equation}
for all integers $n\ge1$.

Our goal in this section is to prove the following proposition:
\begin{prop}
\label{prop:MainProp}Let $k\ge2$, $N\ge1$, $\epsilon>0$. The number
of
\begin{align*}
n&=(n_{1},\dots,n_{k-1})\in \mathbb{Z}^{k-1},\;m=(m_{1},\dots,m_{k-1})\in \mathbb{Z}^{k-1},\\w&=(w_{1},\dots,w_{k})\in\mathbb{Z}^k,\;w'=(w'_{1},\dots,w_{k}')\in\mathbb{Z}^k
\end{align*}
 such that $1\le w_{1},\dots,w_{k}\le N$ are distinct, $1\le w'_{1},\dots,w'_{k}\le N$
are distinct, $$1 \le \left\Vert n \right\Vert _\infty \le N^{1+\epsilon},
1 \le \left\Vert m \right\Vert _\infty \le N^{1+\epsilon},$$ and $$\abs{n \cdot \Delta_{(a_n)}(w) - m \cdot \Delta_{(a_n)}(w')} \le N^{\epsilon}$$
%\begin{align*}
%\biggl|n_{1}\left(a_{w_{1}}-a_{w_{2}}\right) & +\dots+n_{k-1}\left(a_{w_{k-1}}-a_{w_{k}}\right)\\
% & +n_{1}'\left(a_{w_{1}'}-a_{w_{2}'}\right)+\dots+n_{k-1}'\left(a_{w_{k-1}'}-a_{w_{k}'}\right)\biggl|\le N^{\epsilon}
%\end{align*}
is $O\left(N^{2k-1+4k\epsilon}\right).$
\end{prop}

We will begin with an auxiliary
lemma.

\begin{lem}
\label{lem:IntervalCount}Let $I\subseteq\left(0,\infty\right)$ be
a finite interval. Then
\begin{equation}
\#\left\{ n\ge1:\,a_n\in I\right\} \le C\left|I\right|+1,\label{eq:Interval_bound}
\end{equation}
where $C:=\left(a_1\left(1-\frac{1}{c}\right)\right)^{-1}.$
\end{lem}

\begin{proof}
By (\ref{eq:lacunarity}), we have
\[
a_{n+1}-a_n=a_{n+1}\left(1-\frac{a_n}{a_{n+1}}\right)\ge a_{n+1}\left(1-\frac{1}{c}\right)\ge a_1\left(1-\frac{1}{c}\right),
\]
and (\ref{eq:Interval_bound}) follows.
\end{proof}

In the rest of this section, we will follow
the strategy of \cite[Section 2]{RZ2}, adapted to real-valued sequences.

\begin{lem}
\label{lem:region}Let $r\ge1$ be an integer, $C\ge1$, $A_{1}>A_{2}>\dots>A_{r}>0$
real numbers and $b\in\mathbb{R}.$ For any $M\ge1$, the number of
vectors $y=\left(y_{1},\dots,y_{r}\right)\in\mathbb{Z}^{r}$ with
$\left|y_{1}\right|,\dots,\left|y_{r}\right|\le M$ such that
\begin{equation}
\left|y_{1}A_{1}+\dots+y_{r}A_{r}+b\right|\le CA_{1}\label{eq:region}
\end{equation}
is $O\left(CM^{r-1}\right)$.
\end{lem}

\begin{proof}
%We follow the proof of \cite[Lemma 2.1]{RZ2} (where $A_{1},\dots,A_{r},b$
%are assumed to be integers, and the dependence on $C$ is not explicitly
%tracked). Consider the convex region $\Omega \subseteq \mathbb{R}^r $ defined by the intersection
%between the region (\ref{eq:region}) and the cube $\left[-M,M\right]^{r}$;
%we are interested in the number of integer points inside $\Omega$,
%which by the Lipschitz principle \cite{Davenport} is
%\[
%\#\left(\Omega\cap\mathbb{Z}^{r}\right)=\mathrm{vol}\Omega+O\left(M^{r-1}\right).
%\]
%Since the distance between the hyperplanes
%\begin{align*}
%y_{1}A_{1}+\dots+y_{r}A_{r}+b & =-CA_{1}\\
%y_{1}A_{1}+\dots+y_{r}A_{r}+b & =CA_{1}
%\end{align*}
%is $\frac{2CA_{1}}{\sqrt{A_{1}^{2}+\dots+A_{r}^{2}}}\le2C$, we see
%that $\Omega$ is contained in a $r$-dimensional cylinder with height
%$2C$ whose base is an $\left(r-1\right)$-dimensional ball of radius
%$O\left(M\right)$. Hence $\text{vol}\Omega=O\left(CM^{r-1}\right)$,
%and since $C\ge1$ we conclude that
%\[
%\#\left(\Omega\cap\mathbb{Z}^{r}\right)=O\left(CM^{r-1}\right).
%\]
The variables $ y_2,\dots,y_r $ can take at most $ O(M^{r-1}) $ values. Fix $ y_2,\dots,y_r $ and denote \[ \alpha := \frac{y_2 A_2 +\cdots + y_r A_r +b}{A_1}. \] Then, $ y_1 \in [\alpha - C, \alpha + C] $, and therefore $ y_1 $ can take at most $ O(C) $ values.
\end{proof}
\begin{lem}
\label{lem:region_constraint}Let $r\ge2$ be an integer, $C\ge1$,
$z_{1}>z_{2}>\dots>z_{r}>0$ integers and $b\in\mathbb{R}$. For any
$M\ge1$, $d\in\mathbb{Z}$, the number of vectors $y=\left(y_{1},\dots,y_{r}\right)\in\mathbb{Z}^{r}$
with $\left|y_{1}\right|,\dots,\left|y_{r}\right|\le M$ such that
\begin{align}
\left|y_{1}a_{z_{1}}+\dots+y_{r}a_{z_{r}}+b\right| & \le Ca_{z_{1}}\label{eq:region_constraint}\\
y_{1}+\dots+y_{r} & =d\nonumber
\end{align}
is $O\left(CM^{r-2}\right).$
\end{lem}

\begin{proof}
We argue as in \cite[Lemma 2.2]{RZ2} (where $(a_n)$ is
assumed to be an integer valued lacunary sequence and $b$ is assumed
to be integer): by substituting the constraint $y_{1}+\dots+y_{r}=d$
in the inequality (\ref{eq:region_constraint}), we conclude that we
have to bound the number of integer points in the region
\[
\left|y_{1}\left(a_{z_{1}}-a_{z_{r}}\right)+\dots+y_{r-1}\left(a_{z_{r-1}}-a_{z_{r}}\right)+b+da_{z_{r}}\right|\le Ca_{z_{1}};
\]
since $a_{z_{1}}-a_{z_{r}}\ge\left(1-\frac{1}{c}\right)a_{z_{1}},$
we can apply Lemma \ref{lem:region} with $A_{i}=a_{z_{i}}-a_{z_{r}}$
and with the constant on the right-hand-side of (\ref{eq:region})
being equal to $C\cdot\left(1-\frac{1}{c}\right)^{-1}$.
\end{proof}
We will now adapt \cite[Lemma 2.3]{RZ2} to our setting.
\begin{lem}
\label{lem:MainLemma}Let $r\ge1$ be an integer. For any $M\ge1,$
$K\ge1,$ $\epsilon>0$, the number of $\left(y_{1},\dots,y_{r},z_{1},\dots,z_{r}\right)\in\mathbb{Z}^{2r}$
with
\begin{align}
\left|y_{1}\right|,\dots,\left|y_{r}\right| & \le M,\label{eq:cond_range}\\
1\le z_{1},\dots,z_{r} & \le M\hspace{1em}\mathrm{distinct}\nonumber \\
\left(y_{1},\dots,y_{r}\right) & \ne\left(0,\dots,0\right)\nonumber
\end{align}
such that
\begin{align}
\left|y_{1}a_{z_{1}}+\dots+y_{r}a_{z_{r}}\right| & \le K\label{eq:cond_region}\\
y_{1}+\dots+y_{r} & =0\nonumber
\end{align}
is $O\left(K^{r}M^{r-1+\epsilon}\right).$
\end{lem}

\begin{proof}
We prove the lemma by induction on $r$. Clearly, for $r=1$ there
are no vectors satisfying both (\ref{eq:cond_range}), (\ref{eq:cond_region})
(``admissible vectors''), so that the statement of the lemma trivially
holds in this case.

We now assume that the statement holds for $r-1$, and prove it for
$r$. When counting admissible vectors $\left(y_{1},\dots y_{r},z_{1},\dots,z_{r}\right)\in\mathbb{Z}^{2r}$,
we can assume that $y_{i}\ne0$ for all $1\le i\le r$. Indeed, if
there exists $i$ such that $y_{i}=0$, then
\[
\left(y_{1},\dots,y_{i-1},y_{i+1},\dots,y_{r},z_{1},\dots,z_{i-1},z_{i+1},\dots,z_{r}\right)\in\mathbb{Z}^{2\left(r-1\right)}
\]
 are admissible vectors for $r-1$, and therefore by the induction
hypothesis the number of possible $\left(y_{1},\dots,y_{i-1},y_{i+1},\dots,y_{r},z_{1},\dots,z_{i-1},z_{i+1},\dots,z_{r}\right)$
is $O\left(K^{r-1}M^{r-2+\epsilon}\right)$;
since $z_{i}$ can take $O\left(M\right)$ values, the number of admissible
vectors $\left(y_{1},\dots y_{r},z_{1},\dots,z_{r}\right)$ with $y_{i}=0$
is $O\left(K^{r-1}M^{r-1+\epsilon}\right).$

Assume then that $\left(y_{1},\dots y_{r},z_{1},\dots,z_{r}\right)\in\mathbb{Z}^{2r}$
is an admissible vector such that $y_{i}\ne0$ for all $1\le i\le r$;
we can also assume that $z_{1}>z_{2}>\dots>z_{r}$. We will partition
the index set $\left\{ 1,\dots,r\right\} $ to a disjoint union of
sets $B_{1},\dots,B_{l}$ where each set $B_{i}$ will consist of
indices of close-by elements $z_{j}$ in the following sense: $B_{1}$
will consist of $j\in\left\{ 1,\dots,r\right\} $ such that $z_{j}\in\left[z_{1}-\frac{\log M}{\log c},z_{1}\right]$;
if we denote by $j_{2}$ the smallest $j\in\left\{ 1,\dots,r\right\} $
not contained in $B_{1}$, then $B_{2}$ will consist of $j\in\left\{ j_{2},\dots,r\right\} $
such that $z_{j}\in\left[z_{j_{2}}-\frac{\log M}{\log c},z_{j_{2}}\right]$
and so on. If we label by $1=j_{1}<j_{2}<\dots<j_{l}$ the smallest
elements of $B_{1},B_{2},\dots,B_{l}$, then for each $1\le k\le l-1$
we have
\begin{align}
z_{j_{k}},\dots,z_{j_{k+1}-1} & \in\left[z_{j_{k}}-\frac{\log M}{\log c},z_{j_{k}}\right]\nonumber \\
z_{j_{k+1}} & <z_{j_{k}}-\frac{\log M}{\log c}\label{eq:B_distance}
\end{align}
and
\[
z_{j_{l}},\dots,z_{r}\in\left[z_{j_{l}}-\frac{\log M}{\log c},z_{j_{l}}\right].
\]
Since the number of possible partitions of $\left\{ 1,\dots,r\right\} $
into $l$ subsets is $O\left(1\right)$, it is enough to count
the number of admissible vectors which correspond to a given partition.
We distinguish between two cases: $\#B_{l}\ge2$ and $\#B_{l}=1$.

Assume first that $\#B_{l}\ge2$. If we fix $z_{j_{1}},z_{j_{2}},\dots,z_{j_{l}}$,
then each of the remaining numbers $z_{j}$ (there are $r-l$ of them)
belongs to one of the intervals $\left[z_{j_{k}}-\frac{\log M}{\log c},z_{j_{k}}\right],$
and hence can take at most $O\left(\frac{\log M}{\log c}\right)$
values. Thus, $z_{1},z_{2},\dots,z_{r}$ can take at most $O\left(M^{l+\epsilon}\right)$
values; if we fix $z_{1},z_{2},\dots z_{r}$, it is enough to show
that the number of admissible $y_{1},\dots,y_{r}$ is $O\left(K^{r-1}M^{r-l-1}\right).$
Note that by (\ref{eq:B_distance}) and the lacunarity of the sequence
$(a_n)$, we have
\begin{equation}
a_{z_{j_{k}}}/a_{z_{j_{k+1}}}\ge c^{z_{j_{k}}-z_{j_{k+1}}}>M.\label{eq:seperation_bound}
\end{equation}
Fix $z_{1},\dots,z_{r}$ and assume that $y_{1},\dots,y_{r}$ are
admissible. We have
\[
\left|\sum_{j\in B_{1}}y_{j}a_{z_{j}}\right|\le\left|\sum_{j\ge j_{2}}y_{j}a_{z_{j}}\right|+K\le rMa_{z_{j_{2}}}+K < ra_{z_{1}}+K\ll Ka_{z_{1}},
\]
where in the first inequality we used (\ref{eq:cond_region}), and
in the third inequality we used (\ref{eq:seperation_bound}).
By Lemma \ref{lem:region}, we conclude that $y_{1},\dots,y_{j_{2}-1}$
can take at most $O\left(KM^{\#B_{1}-1}\right)$ values.
Now fix $y_{1},\dots,y_{j_{2}-1}$, and set $b=\sum_{j<j_{2}}y_{j}a_{z_{j}}$.
We have
\[
\left|b+\sum_{j\in B_{2}}y_{j}a_{z_{j}}\right|\le\left|\sum_{j\ge j_{3}}y_{j}a_{z_{j}}\right|+K\le rMa_{z_{j_{3}}}+K < ra_{z_{j_{2}}}+K\ll Ka_{z_{j_{2}}},
\]
so that by Lemma \ref{lem:region}, $y_{j_{2}},\dots,y_{j_{3}-1}$
can take at most $O\left(KM^{\#B_{2}-1}\right)$ values.
We repeat this process, and see that $y_{1},\dots,y_{j_{l}-1}$ can
take at most $O\left(K^{l-1}M^{\left(\#B_{1}-1\right)+\dots+\left(\#B_{l-1}-1\right)}\right)$
values, and if we keep them fixed and denote $b=\sum_{j<j_{l}}y_{j}a_{z_{j}}$,
$d=-\sum_{j<j_{l}}y_{j}=\sum_{j\in B_{l}}y_{j}$, then
\[
\left|b+\sum_{j\in B_{l}}y_{j}a_{z_{j}}\right|\ll Ka_{z_{j_{l}}}
\]
and therefore by Lemma \ref{lem:region_constraint} (recall that $\#B_{l}\ge2$),
$y_{j_{l}},\dots,y_{r}$ can take at most $O\left(KM^{\#B_{l}-2}\right)$
values. We see that $y_{1},\dots,y_{r}$ can take at most
\[
O\left(K^{l}M^{\left(\#B_{1}-1\right)+\dots+\left(\#B_{l-1}-1\right)+\left(\#B_{l}-2\right)}\right)=O\left(K^{r-1}M^{r-l-1}\right)
\]
values.

Assume now that $\#B_{l}=1$, so that $j_{l}=r$. By the above argument
$z_{1},\dots,z_{r-1}$ can take at most $O\left(M^{l-1+\epsilon}\right)$
values. We keep $z_{1},\dots,z_{r-1}$ fixed, and again, by the argument
above $y_{1},\dots,y_{r-1}$ can take at most $$O\left(K^{l-1}M^{\left(\#B_{1}-1\right)+\dots+\left(\#B_{l-1}-1\right)}\right)=O\left(K^{l-1}M^{r-l}\right)$$
values. Assume that $y_{1},\dots,y_{r-1},z_{1},\dots,z_{r-1}$ are
fixed. Then $y_{r}=-y_{1}-\dots-y_{r-1}$ is uniquely determined,
and since by our assumption it is non-zero and integer it satisfies
$\left|y_{r}\right|\ge1$. Let us bound the number of possible values
of $z_{r}$: denote \[\alpha:=\frac{-y_{1}a_{z_{1}}-\dots-y_{r-1}a_{z_{r-1}}}{y_{r}}.\]
Then
\[
\left|a_{z_{r}}-\alpha\right|\le\frac{K}{\left|y_{r}\right|}\le K
\]
so that
\begin{equation*}
a_{z_{r}}\in\left[\alpha-K,\alpha+K\right].\label{eq:possible_interval}
\end{equation*}
Hence, by Lemma \ref{lem:IntervalCount}, $z_{r}$ can take at most
$O\left(K\right)$ values, and hence the number
of admissible $y_{1},\dots,y_{r},z_{1},\dots,z_{r}$ is $O\left(K^{r}M^{r-1+\epsilon}\right).$
\end{proof}
We would like to prove a generalization of Lemma \ref{lem:MainLemma}
to vectors $\left(y_{1},\dots,y_{r},z_{1},\dots,z_{r}\right)\in\mathbb{Z}^{2r}$
consisting of non-distinct $z_{1},\dots,z_{r}$. We will require a
non-degeneracy condition that we now describe.

Given a vector $v=\left(y_{1},\dots,y_{r},z_{1},\dots,z_{r}\right)\in\mathbb{Z}^{2r}$,
for any $1\le i\le r$ we let $A\left(i\right):=\left\{ 1\le j\le r:\,z_{j}=z_{i}\right\} $.
We say that the vector $ v $ is degenerate if for any $1\le i\le r$
we have $\sum_{j\in A\left(i\right)}y_{j}=0$, and we say that $v$
is non-degenerate otherwise.
\begin{lem}
\label{lem:NonDegenerate}Let $r\ge1$ be an integer. For any $M\ge1,$
$K\ge1,$ $\epsilon>0$, the number of non-degenerate $\left(y_{1},\dots,y_{r},z_{1},\dots,z_{r}\right)\in\mathbb{Z}^{2r}$
with
\begin{align}
\left|y_{1}\right|,\dots,\left|y_{r}\right| & \le M,\label{eq:cond_range_2}\\
1\le z_{1},\dots,z_{r} & \le M\nonumber
\end{align}
such that
\begin{align}
\left|y_{1}a_{z_{1}}+\dots+y_{r}a_{z_{r}}\right| & \le K\label{eq:cond_region_2}\\
y_{1}+\dots+y_{r} & =0\nonumber
\end{align}
is $O\left(K^{r}M^{r-1+\epsilon}\right).$
\end{lem}

\begin{proof}
For each vector $ v=\left(y_{1},\dots,y_{r},z_{1},\dots,z_{r}\right)\in\mathbb{Z}^{2r}$
satisfying (\ref{eq:cond_range_2}), (\ref{eq:cond_region_2}) (``admissible
vector''), the corresponding sets $A\left(1\right),\dots,A\left(r\right)$
induce a partition of the index set $\left\{ 1,\dots,r\right\} $
into disjoint union of sets $A_{1},\dots,A_{l}$ (which are exactly
the sets $A\left(1\right),\dots,A\left(r\right)$ without repetitions).
Since the total number of partitions of $\left\{ 1,\dots,r\right\} $
into $l$ subsets is $O\left(1\right)$, we can count only admissible
vectors corresponding to a given partition.

For each $1\le i\le l$, let $\tilde{y_{i}}=\sum_{j\in A_{i}}y_{j}$
and let $\tilde{z}_{i}=z_{j}$ for $j\in A_{i}.$ Then $\left|\tilde{y}_{1}\right|,\dots,\left|\tilde{y}_{l}\right|\ll M$,
$\tilde{y}_{1}+\dots+\tilde{y}_{l}=0,$ by the non-degeneracy condition
$\left(\tilde{y}_{1},\dots,\tilde{y}_{l}\right)\ne\left(0,\dots,0\right)$,
$1\le\tilde{z}_{1},\dots,\tilde{z}_{l}\le M$ are distinct, and $\left|\tilde{y}_{1}a_{\tilde{z}_{1}}+\dots+\tilde{y}_{l}a_{\tilde{z}_{l}}\right|\le K$.
Hence we can apply Lemma \ref{lem:MainLemma} and deduce that $\tilde{y}_{1},\dots,\tilde{y}_{l},\tilde{z}_{1},\dots,\tilde{z}_{l}$
can take at most $O\left(K^{l}M^{l-1+\epsilon}\right)$
values.

We now fix $\tilde{ v}=\left(\tilde{y}_{1},\dots,\tilde{y}_{l},\tilde{z}_{1},\dots,\tilde{z}_{l}\right)$,
and count the number of possible values of $y_{1},\dots,y_{r},z_{1},\dots,z_{r}$
%with a corresponding partition $A_{1},\dots,A_{l}$
which map to $\tilde{ v}$.
For each $1\le i\le l,$ all values of $z_{j},$ $j\in A_{i}$ are
equal to $\tilde{z}_{i}$, so $z_{1},\dots,z_{r}$ are completely
determined by $\tilde{ v}.$ Moreover, for each $1\le i\le l$
we have $\tilde{y}_{i}=\sum_{j\in A_{i}}y_{j}$, and for fixed $\tilde{y}_{i}$
the number of solutions to this equation is $O\left(M^{\#A_{i}-1}\right).$
Hence $y_{1},\dots,y_{r}$ can take at most $O\left(M^{\left(\#A_{1}-1\right)+\dots+\left(\#A_{l}-1\right)}\right)=O\left(M^{r-l}\right)$
values. We conclude that $ v $ can take at most $O\left(K^{r}M^{r-1+\epsilon}\right)$ values 
as claimed.
\end{proof}
We are now in the position to prove Proposition \ref{prop:MainProp}.
\begin{proof}[Proof of Proposition \ref{prop:MainProp}]
Let $r=2k$, and let
\begin{align*}
z_{1} & =w_{1},\dots,z_{k}=w_{k},\\
z_{k+1} & =w_{1}',\dots,z_{2k}=w_{k}',\\
y_{1} & =n_{1},y_{2}=n_{2}-n_{1},\dots,y_{k-1}=n_{k-1}-n_{k-2},y_{k}=-n_{k-1},\\
y_{k+1} & =-m_{1},y_{k+2}=m_{1}-m_{2},\dots,y_{2k-1}=m_{k-2}-m_{k-1},y_{2k}=m_{k-1}.
\end{align*}
We see that if $n=(n_{1},\dots,n_{k-1}),m=(m_{1},\dots,m_{k-1}) ,w=(w_{1},\dots,w_{k}),w'=(w'_{1},\dots,w_{k}')$
satisfy the conditions of Proposition \ref{prop:MainProp}, then the
vector $ v =\left(y_{1},\dots,y_{r},z_{1},\dots,z_{r}\right)$
satisfies (\ref{eq:cond_range_2}), (\ref{eq:cond_region_2}) with $M=2N^{1+\epsilon}$,
$K=N^{\epsilon}$ together with the additional conditions
\begin{align*}
\left(y_{1},\dots,y_{k}\right) & \ne\left(0,\dots,0\right),\\
\left(y_{k+1},\dots,y_{2k}\right) & \ne\left(0,\dots0\right),\\
y_{1}+\dots+y_{k} & =0,\\
z_{1},\dots,z_{k} & \text{\ensuremath{\hspace{1em}}distinct,}\\
z_{k+1},\dots,z_{2k+1} & \text{\ensuremath{\hspace{1em}}distinct}.
\end{align*}
It is therefore sufficient to bound the number of such ``admissible''
vectors $ v .$

By Lemma \ref{lem:NonDegenerate}, for any $\eta>0$, the number of
non-degenerate admissible vectors is $O\left(K^{r}M^{r-1+\eta}\right)$,
which upon taking $\eta$ sufficiently small is also $O\left(N^{2k-1+4k\epsilon}\right)$.
It remains to count the number of degenerate admissible vectors.

Assume that $ v $ is degenerate, and denote by $s$ the number of variables among $z_{1},\dots,z_{k}$
which are equal to one of the variables $z_{k+1},\dots,z_{2k}$ (clearly
$s\ge1$, since $ v $ is degenerate and $\left(y_{1},\dots,y_{r}\right)\ne\left(0,\dots,0\right)$).
To simplify the notation, we can assume that $z_{1}=z_{k+1},\dots,z_{s}=z_{k+s}$.
Hence, the sets $A_{1},\dots,A_{l}$ defined in the proof of Lemma
\ref{lem:NonDegenerate} are exactly
\[
\left\{ 1,k+1\right\} ,\dots,\left\{ s,k+s\right\} ,\left\{ s+1\right\} ,\dots,\left\{ k\right\} ,\left\{ k+s+1\right\} ,\dots,\left\{ 2k\right\} ,
\]
so that $l=2k-s$. Since $ v $ is degenerate, we have
\begin{equation}
\begin{cases}
y_{i}+y_{k+i}=0 & 1\le i\le s\\
y_{i}=0 & s+1\le i\le k,\,k+s+1\le i\le2k.
\end{cases}\label{eq:degeneracy}
\end{equation}
Given a partition $A_{1},\dots,A_{l}$ (which can be assumed to be
 fixed), there are exactly $l=2k-s$ distinct variables $z_{i}$,
and hence $z_{1},\dots,z_{k}$ can take at most $M^{2k-s}$ values.
Given $y_{2},\dots,y_{s},$ the variables
\[
y_{k+2},\dots,y_{k+s},y_{s+1},\dots,y_{k},y_{k+s+1},\dots,y_{2k}
\]
are determined by (\ref{eq:degeneracy}), whereas $y_{1}$ is determined
by the condition $y_{1}+\dots+y_{k}=0$, and then $y_{k+1}$ is determined
by (\ref{eq:degeneracy}). Hence, the variables $y_{1},\dots,y_{2k}$
can take at most $O\left(M^{s-1}\right)$ values, and the total number
of degenerate vectors $ v $ is at most $O\left(M^{2k-1}\right)=O\left(N^{2k-1+2k\epsilon}\right).$
\end{proof}

\section{Proof of Theorem \ref{thm:MainThm}}
\label{sec:proof}
%We define the $k$-level correlation function as follows. For $x=(x_1, x_2, \cdots, x_k),$ satisfying $1\le x_i\le N$, denote
%\[ \Delta(x)=(a(x_1)-a(x_2), \cdots, a(x_{k-1})-a(x_k).\]
%Take a smooth, compactly supported function $f\in C_c^{\infty}(\mathbb{R}^{k-1})$, and set
%\[F_N(y)=\sum_{m\in\mathbb{Z}^{k-1}}f(N(m+y))\]
%The $k$-level correlation function is defined as
%\begin{align}\label{k-level-func}
%    R_k(f,N)(\alpha)=\frac{1}{N}\sum_{x_i\le N}^*F_N(\alpha\Delta(x)),
%\end{align}
%where $\sum^*$ denotes the sum over distinct components $x_i\neq x_j$ if $i\neq j$.
%\begin{theorem}
%For all $k\ge 2$, the k-level correlation function $R_k(f,N)(\alpha)$ converges to $\int f(x)~dx,$ for all test functions $f\in C_c^{\infty}(\mathbb{R}^{k-1})$, and almost all $\alpha\in\mathbb{R}.$
%\begin{remark}
%It suffices to restrict $\alpha$ to a fixed finite interval and to consider a smooth average. For $\rho\in C_c^{\infty}(\mathbb{R}^{k-1})$, such that $\int_{\mathbb{R}}\rho(\alpha)~d\alpha=1.$ For a real-valued function $f$, we define the smooth average as
%\[\langle X\rangle=\int_{\mathbb{R}}X(\alpha)\rho(\alpha)~d\alpha.\]
%\end{remark}
%\end{theorem}
%We assume that $(a_n)_{n=1}^\infty$ is a lacunary sequence of positive
%\emph{real} numbers satisfying
%\begin{equation}
%a_{n+1}\ge ca_n\label{eq:lacunarity}
%\end{equation}
%for all integers $n\ge1$, where $c>1$.

Fix $k\ge2$. We will now turn to prove our main theorem, estimating the variance of the $k$-level correlation sums $$R_k(f,N)(\alpha) := R_{k}\left(f,\left(\alpha a_{n}\right),N\right)$$ using Proposition \ref{prop:MainProp}. %We can clearly assume that $\alpha \in I$ where $I$ is a fixed finite interval.
It will be technically easier to work with smooth averages; we therefore fix a smooth, compactly supported, non-negative weight function  $\rho:\mathbb{R}\to \mathbb{R}$. %such that $\int_{\mathbb{R}}\rho(\alpha)~d\alpha=1.$

\subsection{Variance}
We would like to show that the variance of $R_k(f,N)(\alpha)$ w.r.t. $ \alpha $ is small. The fact that the expectation of $R_k(f,N)(\alpha)$ is asymptotic to $\int_{\mathbb{R}^{k-1}}f(x)~dx$ can be shown in a similar way; we omit the proof since it is not required for the proof of Theorem \ref{thm:MainThm}.
%asymptotic to the average $\int_{\mathbb{R}}f(x)~dx$ under the assumptions \eqref{}, and \eqref{}.

Let
\begin{align}
V(R_k(f,N),\rho)&=\int_{\mathbb{R}}\bigabs{R_k(f,N)(\alpha)-C_k(N)\int_{\mathbb{R}^{k-1}}f(x)~dx}^2\rho(\alpha)~d\alpha\nonumber
\end{align}
denote the variance of $R_k(f,N)(\alpha)$, where
\begin{align*}C_k(N):= \left(1-\frac{1}{N}\right)\cdots\left(1-\frac{k-1}{N}\right)=1+O\left(\frac{1}{N}\right).
\end{align*}

\begin{prop}
We have
\begin{align}\label{eq:variance_bound}
V(R_k(f,N),\rho) =  O(N^{-1+\eta})
\end{align}
for all $\eta >0$.
\end{prop}

\begin{proof}

By the Poisson summation formula, the $k$-level correlation sum is
\begin{align}\label{eq:PoissonSummation}
	R_k(f,N)(\alpha)= C_k(N)\hat{f}(0)+\frac{1}{N^k}\sum_{0_{k-1}\neq n\in\mathbb{Z}^{k-1}}\hat{f}\left(\frac{n}{N}\right)\sum_{\substack{x=(x_1,\dots,x_k)\\
		1\le x_{1},\dots,x_{k}\le N\\
		x_{1},\dots,x_{k}\,\text{distinct}}}e(n \alpha \cdot\Delta_{(a_n)}(x)),
\end{align}
where use the standard notation $e(z)=e^{2 \pi i z}$.

Using \eqref{eq:PoissonSummation}, we have
\begin{align}\label{Var}
    V(R_k(f,N),\rho)=&\frac{1}{N^{2k}}\sum_{\substack{0_{k-1}\neq n,m\in\mathbb{Z}^{k-1}\\ }}\hat{f}\left(\frac{n}{N}\right)\overline{\hat{f}\left(\frac{m}{N}\right)}\sum^*_{x,y}\hat{\rho}(n\cdot\Delta_{(a_n)}(x)-m\cdot\Delta_{(a_n)}(y)),
\end{align}
where the summation in $ \sum\limits^* $ is over $ x=(x_1,\dots,x_k) , y=(y_1,\dots,y_k)$ such that $ 1\le x_1,\dots,x_k \le N $ are distinct, and $ 1\le y_1,\dots,y_k \le N $ are distinct.

Fix $ \epsilon>0 $. We break the sums over $n$ and $m$ into ranges $\max\{\left\Vert n \right\Vert_\infty, \left\Vert m \right\Vert_\infty\}> N^{1+\epsilon}$ and $\max\{\left\Vert n \right\Vert_\infty, \left\Vert m \right\Vert_\infty\}\le N^{1+\epsilon}$. We assume $\max\{\left\Vert n \right\Vert_\infty, \left\Vert m \right\Vert_\infty\}=\left\Vert n \right\Vert_\infty$, since the other case follows similarly. In the range
$0<\left\Vert n \right\Vert_\infty \le N^{1+\epsilon}$, we use the bound $\hat{f}\ll 1$, and in the range
$\left\Vert n \right\Vert_\infty > N^{1+\epsilon}$, we use $\hat{f}(x)\ll \left\Vert x \right\Vert_\infty^{-R}$ for arbitrarily large $ R>0 $ and $\rho \ll 1$. This gives that \eqref{Var} is bounded by
\begin{align} \label{var1}
  &\frac{1}{N^{2k}}\sum_{\substack{0<\left\Vert m \right\Vert_\infty\le N^{1+\epsilon}\\\left\Vert n \right\Vert_\infty>N^{1+\epsilon}}}\left\Vert \frac{n}{N} \right\Vert_\infty^{-R}\sum^*_{x,y}1+ \frac{1}{N^{2k}}\sum_{\left\Vert n \right\Vert_\infty,\left\Vert m \right\Vert_\infty> N^{1+\epsilon}}\left\Vert \frac{n}{N} \right\Vert_\infty^{-R}\left\Vert \frac{m}{N} \right\Vert_\infty^{-R}\sum^*_{x,y}1\nonumber\\&+
  \frac{1}{N^{2k}}\sum_{0< \left\Vert n \right\Vert_\infty, \left\Vert m \right\Vert_\infty\le N^{1+\epsilon}}\sum^*_{x,y}\bigabs{\hat{\rho}\left(n\cdot\Delta_{(a_n)}(x)-m\cdot\Delta_{(a_n)}(y)\right)}.
  \end{align}
The second term in \eqref{var1} is at most
$ O(N^{2(k-1)(1+\epsilon)-2 \epsilon R}).$ Similarly, the first term in \eqref{var1} is
$O(N^{2(k-1)(1+\epsilon)-\epsilon R}).$
In order to estimate the third term in \eqref{var1}, we further break the sum into the ranges $\abs{n\cdot\Delta_{(a_n)}(x)-m\cdot\Delta_{(a_n)}(y)}\le N^\epsilon$, and $\abs{n\cdot\Delta_{(a_n)}(x)-m\cdot\Delta_{(a_n)}(y)}> N^\epsilon$. Using the bound $\hat{\rho}\ll 1,$ the total contribution of the third term in \eqref{var1} restricted to $\abs{n\cdot\Delta_{(a_n)}(x)-m\cdot\Delta_{(a_n)}(y)}\le N^\epsilon$ is
$ \ll \frac{1}{N^{2k}} A(N,\epsilon),$ where
\begin{align*}
   A(N,\epsilon) = \#\big\{&1\le \left\Vert n \right\Vert_\infty, \left\Vert m \right\Vert_\infty \le N^{1+\epsilon}, x=(x_1,\dots,x_k), 1\le x_i\le N \; \mathrm{distinct}, \\ &y=(y_1,\dots,y_k), 1\le y_i\le N \; \mathrm{distinct}, |n\cdot\Delta_{(a_n)}(x)-m\cdot\Delta_{(a_n)}(y)|\le N^\epsilon\big\}.
\end{align*}
Taking $w=x,w'=y$ in Proposition \ref{prop:MainProp}, we get that \[\frac{1}{N^{2k}} A(N,\epsilon)\ll N^{-1+4k\epsilon}.\]

For the second range $\abs{n\cdot\Delta_{(a_n)}(x)-m\cdot\Delta_{(a_n)}(y)} > N^\epsilon$, we have
\[\abs{\hat{\rho}(n\cdot\Delta_{(a_n)}(x)-m\cdot\Delta_{(a_n)}(y))}\ll \abs{n\cdot\Delta_{(a_n)}(x)-m\cdot\Delta_{(a_n)}(y)}^{-R} < N^{-\epsilon R}\] for arbitrarily large $R>0$. This gives that contribution of the third term of \eqref{var1} restricted to this range is at most
\[\sum_{0< \left\Vert n \right\Vert_\infty, \left \Vert m \right\Vert_\infty \le N^{1+\epsilon}} N^{-\epsilon R}\le N^{2(k-1)(1+\epsilon)-\epsilon R}.\]

Finally, the bound \eqref{eq:variance_bound} follows from the above estimates upon taking $ \epsilon = \frac{\eta}{4k} $  and $ R $ sufficiently large.
\end{proof}

\subsection{Almost sure convergence}
Having proved the variance bound \eqref{eq:variance_bound}, the almost sure convergence of the $k$-level correlation sums to $\int_{\mathbb{R}^{k-1}}f\left(x\right)\,dx$ follows from a standard argument, as formulated in a general setting in the following proposition taken from \cite{YT}.

\begin{prop}[{\cite[Proposition 7.1]{YT}}]
	Fix $ k\ge2 $, $ J\subset \mathbb{R} $ a bounded interval, and a sequence $ c_k(N) $ such that $ \lim_{N\to \infty}c_k(N)= 1 $. Let $ (\vartheta_n(\alpha))_{n\ge 1}\;(\alpha \in J) $ be a parametric family of sequences such that the map $ \alpha \mapsto \vartheta_n(\alpha) $ is continuous for each fixed $ n\ge1 $. Assume that there exists $ \delta>0 $ such that for any compactly supported, real
	valued, smooth function $f:\mathbb{R}^{k-1}\to\mathbb{R}$ we have
\begin{align}\label{non_smooth_variance_bound} \int_J\left\vert R_{k}\left(f,(\vartheta_n(\alpha)),N\right)-c_{k}\left(N\right)\int_{\mathbb{R}^{k-1}}f\left(x\right)\,dx\right\vert^{2}\,d\alpha = O(N^{-\delta})
\end{align}
as $ N\to \infty $. Then for almost all $ \alpha \in J $, the sequence $(\vartheta_n(\alpha))_{n\ge 1}$ has Poissonian $ k$-point correlation.
\end{prop}

Indeed, we can clearly assume that $\alpha \in J$ where $J$ is a fixed finite interval and take $\rho$ such that $\rho \ge \mathbf{1}_J$. Let $ \vartheta_n(\alpha) = \alpha a_n $ and  $c_k(N)=C_k(N)$; the bound \eqref{non_smooth_variance_bound} follows from \eqref{eq:variance_bound}, since
\begin{align}
	\int_J\left\vert R_{k}\left(f,(\vartheta_n(\alpha)),N\right)-c_{k}\left(N\right)\int_{\mathbb{R}^{k-1}}f\left( x \right)\,d x\right \vert ^{2}\,d \alpha \le V(R_k(f,N),\rho).
\end{align}
Thus, Theorem \ref{thm:MainThm} follows.

 \end{document}